\documentclass[11pt]{article}
\usepackage[a4paper,margin=1in]{geometry}
\usepackage[T1]{fontenc}
\usepackage[utf8]{inputenc}
\usepackage{lmodern}
\usepackage{amsmath,amssymb,amsthm,mathtools,mathrsfs,bm}
\usepackage{enumitem}
\usepackage[authoryear,round]{natbib}
\usepackage{microtype}
\usepackage{hyperref}
\usepackage[nameinlink,capitalize]{cleveref}
\hypersetup{colorlinks=true,linkcolor=blue,citecolor=blue,urlcolor=blue}

\newtheorem{theorem}{Theorem}[section]
\newtheorem{proposition}[theorem]{Proposition}
\newtheorem{corollary}[theorem]{Corollary}
\newtheorem{lemma}[theorem]{Lemma}
\theoremstyle{definition}
\newtheorem{definition}[theorem]{Definition}
\theoremstyle{remark}
\newtheorem{remark}[theorem]{Remark}

\newcommand{\R}{\mathbb R}
\newcommand{\M}{\mathcal M}
\newcommand{\DM}{\mathbb D_M}
\newcommand{\hookc}{\mathrel{\hookrightarrow\!\hookrightarrow}}

\title{Derivative Embeddings in Vector-Valued Morrey Spaces}
\author{Rishad Shahmurov\\
\small Cellular Products Research and Development, 595 East Crossville Road, Roswell, GA 30075, USA\\
\small \texttt{rshahmurov@crimson.ua.edu}
\and
Veli Shahmurov\thanks{Corresponding author}\\
\small Antalya Bilim University, Antalya, Turkey\\
\small \texttt{veli.sahmurov@antalya.edu.tr}}
\date{}

\begin{document}
\maketitle

\begin{abstract}
We develop an intermediate-derivative theory for Banach-valued evolution classes with Morrey control in time.  Let $E_0\hookrightarrow E$ be Banach spaces, $m\ge2$, $1<p<\infty$, and $0\le\lambda<1$.  If $u\in\mathcal M^{p,\lambda}(E_0)$ and $u^{(m)}\in\mathcal M^{p,\lambda}(E)$, then every intermediate derivative satisfies
\[
u^{(j)}\in\mathcal M^{p,\lambda}\bigl((E_0,E)_{j/m,\infty}\bigr),\qquad 0<j<m,
\]
with no UMD, Fourier type, or operator-valued multiplier assumption.  The proof is based on an exact rational-frequency splitting of $D^j$, the Peetre $K$-functional, and the Hardy--Littlewood maximal operator on Morrey spaces.  On the line the estimate has the scale-sharp multiplicative form.  On finite intervals we construct a graph-Morrey extension and obtain the same embedding.  The weak fine index is optimal for arbitrary Banach couples.  Combining the derivative theorem with the exact Morrey trace theorem yields an exact onto trace formula for every $u^{(j)}$, together with mixed H\"older/interpolation embeddings and a trace-powered $p$-to-$q$ Morrey gain.  In the strict-gap regime, the latter estimate holds for every continuously embedded Banach couple if and only if the target Morrey parameter satisfies $\mu\ge\lambda$; necessity already fails on one fixed densely and compactly embedded Hilbert couple.  If $E_0\hookc E$, these embeddings become compact both in the same Morrey norm after an arbitrarily small loss of interpolation smoothness and in time-H\"older spaces above the derivative-trace threshold.  For positive selfadjoint Hilbert scales we further improve the weak interpolation target to the fractional domain $D(B^{1-j/m})$ and show that this exponent is sharp.  This refinement has a concrete PDE payoff: in third-order nonlinear acoustics it makes the mixed acceleration regularity an automatic consequence of endpoint control, removing a separate space--time interpolation rung from JMGT well-posedness and relaxation-limit estimates and, for periodic trajectories, recovering the natural $\sqrt{\tau}$ scaling.
\end{abstract}

\noindent\textbf{Keywords.} vector-valued Morrey spaces; intermediate derivatives; real interpolation; Lions--Peetre embeddings; trace spaces; compact embeddings.\\
\textbf{MSC 2020.} Primary 46E35, 46B70; Secondary 46M35, 35B65, 47D06.

\section{Introduction}

The classical intermediate-derivative theorem is one of the basic links between evolution regularity and interpolation.  In its Hilbert-space form, developed in the Lions--Magenes theory, information on a function and its highest time derivative places every intermediate derivative at the corresponding interpolation level; see \cite{LionsMagenes1972}.  Banach-valued Sobolev--Lions embeddings and operator-valued variants have subsequently been studied by Fourier multiplier methods, often under geometric or operator-theoretic hypotheses on the underlying spaces; see, for example, \cite{Shakhmurov2006,Shakhmurov2008}.  Related Morrey-space embeddings with operator coefficients were considered in \cite{RagusaShakhmurov2018}.

The purpose of this paper is different.  We ask what can be proved from \emph{time-Morrey control alone}, for an arbitrary Banach couple $E_0\hookrightarrow E$, without UMD assumptions and without representing the couple by the domain of a sectorial operator.  The answer has the same structural feature as the exact Morrey trace theorem: the geometry-free target has weak real-interpolation fine index $\infty$.

For $m\in\mathbb N$, $m\ge2$, define
\[
\DM^{m,p,\lambda}(I;E_0,E)
:=\{u\in\M^{p,\lambda}(I;E_0):D_t^m u\in\M^{p,\lambda}(I;E)\}.
\]
Our first result is
\[
\boxed{
D_t^j:\DM^{m,p,\lambda}(I;E_0,E)
\longrightarrow
\M^{p,\lambda}\bigl(I;(E_0,E)_{j/m,\infty}\bigr),
\quad 0<j<m.
}
\]
The proof is not a vector-valued singular-integral argument.  On the full line we split the multiplier $(i\xi)^j$ into two rational pieces at an arbitrary scale $h$: a low-frequency part acting on $u$ and a high-frequency part acting on $D_t^m u$.  Their scalar kernels are integrable and uniformly dominated by Hardy--Littlewood averages.  Choosing the $K$-functional parameter $r=h^{-m}$ cancels the two scale factors exactly.  The boundedness of the scalar maximal operator on Morrey spaces then finishes the argument.  In particular, the theorem is insensitive to the Banach geometry of $E_0$ and $E$.  This avoidance of vector-valued Calder\'on--Zygmund theory is substantive: for convolution Calder\'on--Zygmund operators, boundedness on one genuine Morrey scale already recovers the corresponding $L^p$ boundedness and hence the classical Banach-space obstruction; see \cite{ShahmurovCZ2026}.

On the full line the estimate is scale sharp:
\[
\|D_t^j u\|_{\M^{p,\lambda}((E_0,E)_{j/m,\infty})}
\lesssim
\|u\|_{\M^{p,\lambda}(E_0)}^{1-j/m}
\|D_t^m u\|_{\M^{p,\lambda}(E)}^{j/m}.
\]
A bounded extension theorem transfers the result to finite intervals.

Already at this stage there is a concrete PDE reason to care about the endpoint-to-middle mechanism.  Third-order nonlinear acoustic equations of Jordan--Moore--Gibson--Thompson (JMGT) type contain simultaneously a relaxation term $\tau u_{ttt}$ and a strong damping term $-b\Delta u_t$, while their energy and nonlinear estimates require one spatial derivative of the acceleration $u_{tt}$.  Thus endpoint control naturally raises the question whether the middle regularity $u_{tt}\in H_x^1$ must be estimated separately.  The Hilbert-scale refinement below shows that it does not: in a selfadjoint elliptic scale, the two endpoints $u_t\in H_x^2$ and $u_{ttt}\in L_x^2$ force $u_{tt}\in H_x^1$.  Section~\ref{sec:applications} displays the full JMGT--Westervelt and JMGT--Kuznetsov equations, identifies the corresponding operator quantities, and shows that the periodic estimate also produces the natural $\sqrt{\tau}$ relaxation weight.

The derivative theorem interacts particularly well with the exact trace theory.  Put
\[
a:=\frac{1-\lambda}{p},\qquad \vartheta:=1-a,
\qquad X_s:=(E_0,E)_{s,\infty}.
\]
For $0<\lambda<1$ and $0\le j<m$, the function $v=D_t^j u$ belongs to the first-order Morrey evolution class with strong space $X_{j/m}$ and weak space $X_{(j+1)/m}$.  Consequently the derivative trace is exact:
\[
\operatorname{Tr}_0D_t^j\DM^{m,p,\lambda}(0,T;E_0,E)
=X_{\tau_j},
\qquad
\tau_j:=\frac{j+a}{m}.
\]
Both the level $\tau_j$ and the weak fine index are sharp.  Moreover, for $\tau_j<\sigma<(j+1)/m$,
\[
D_t^j u\in C^{0,m(\sigma-\tau_j)}([0,T];X_\sigma).
\]
Finite interpolation indices are obtained below the same line with an arbitrarily small loss in the time H\"older exponent.

The scale-changing Morrey gain has an exact parameter threshold.  Let
\[
a=\frac{1-\lambda}{p},\qquad
b=\frac{1-\mu}{q},\qquad
\sigma=\frac{j+a-b}{m},
\]
and assume $0<b<a$ and $q<\infty$.  For each fixed $1\le r\le\infty$, the estimate
\[
D_t^j:\DM^{m,p,\lambda}(0,T;E_0,E)
\longrightarrow
\M^{q,\mu}\bigl(0,T;(E_0,E)_{\sigma,r}\bigr)
\]
holds for every continuously embedded Banach couple $E_0\hookrightarrow E$ if and only if
\[
\boxed{\mu\ge\lambda}.
\]
The positive direction combines the critical derivative estimate with the uniform derivative-trace estimate.  Necessity is witnessed by a separated pulse train in one fixed diagonal Hilbert couple, so the restriction is not caused by unfavorable Banach geometry or by the weak interpolation index.

The compactness consequences split into two natural regimes.  If $E_0\hookc E$, then for every
\[
\frac jm<\sigma<\frac{j+1}{m}
\]
the map $u\mapsto D_t^j u$ is compact from bounded subsets of the higher-order Morrey graph class into the \emph{same} Morrey norm with values in $(E_0,E)_{\sigma,q}$.  Above the derivative-trace level $\tau_j$, one obtains compactness in time-H\"older spaces.  These results are applications of the exact compactness theory for first-order Morrey evolution spaces developed in \cite{ShahmurovCompact2026}.

Finally we show what additional Hilbert geometry buys.  If $B\ge I$ is positive selfadjoint on a Hilbert space $H$, the weak real-interpolation space can be replaced by the sharper fractional domain:
\[
D_t^j u\in \M^{p,\lambda}\bigl(D(B^{1-j/m})\bigr).
\]
The exponent $1-j/m$ cannot be improved.  Thus the paper separates the geometry-free Banach endpoint from the stronger Hilbert-scale theorem instead of hiding the geometric input inside a multiplier argument.

The acoustic application announced above is developed in Section~\ref{sec:applications}, together with its connection to nonlinear-ultrasound and acoustic-tomography models; see \cite{KaltenbacherPeriodic2025,MelianiSaidHouari2025,KaltenbacherTomography2025}.

The earlier operator-valued Morrey literature contains related derivative embeddings under multiplier and positivity hypotheses; see in particular \cite{RagusaShakhmurov2018,Shakhmurov2006,Shakhmurov2008}.  The arguments below are independent of those results and isolate a geometry-free weak real-interpolation conclusion.  We do not claim an arbitrary-Banach finite fine-index theorem at the critical intermediate level $j/m$: Proposition~\ref{prop:weak-index-sharp} shows that such a replacement is false in general.  On finite intervals, however, the exact derivative trace combines with the critical derivative estimate to give a legitimate scale-changing $p$-to-$q$ Morrey consequence strictly between the critical derivative level and the trace level.  Proposition~\ref{prop:mu-threshold} shows that, in the strict-gap regime, the condition $\mu\ge\lambda$ is not merely sufficient but necessary for a theorem valid for all Banach couples.  Likewise, none of the compactness theorems here uses weak compactness of bounded sets in a Morrey space; all compactness is routed through the same-Morrey Aubin--Lions theorem and strict interpolation gaps.

\section{Preliminaries}

Let $I$ denote either $\R$ or a finite interval $(0,T)$.  For a Banach space $X$, $1<p<\infty$, and $0\le\lambda<1$, set
\[
\|f\|_{\M^{p,\lambda}(I;X)}
:=\sup_{J\subset I}|J|^{-\lambda/p}
\left(\int_J\|f(t)\|_X^p\,dt\right)^{1/p},
\]
where the supremum is over bounded nonempty intervals $J\subset I$.

For a compatible Banach couple $(E_0,E)$, let
\[
K(r,x):=\inf_{x=x_0+x_1}
\bigl(\|x_0\|_{E_0}+r\|x_1\|_E\bigr).
\]
For $0<s<1$,
\[
\|x\|_{(E_0,E)_{s,\infty}}
=\sup_{r>0}r^{-s}K(r,x).
\]
We write
\[
X_s:=(E_0,E)_{s,\infty},\qquad X_0:=E_0,\quad X_1:=E.
\]

\begin{definition}
Let $m\ge1$.  We define
\[
\DM^{m,p,\lambda}(I;E_0,E)
:=\{u\in\M^{p,\lambda}(I;E_0):D_t^m u\in\M^{p,\lambda}(I;E)\},
\]
where $D_t^m u$ is taken in the sense of $E$-valued distributions.  The graph norm is
\[
\|u\|_{\DM^{m,p,\lambda}}
:=\|u\|_{\M^{p,\lambda}(E_0)}
+\|D_t^m u\|_{\M^{p,\lambda}(E)}.
\]
\end{definition}

We shall use the scalar Hardy--Littlewood maximal operator $M$.  Its boundedness on Morrey spaces,
\[
\|Mf\|_{\M^{p,\lambda}(\R)}\le C_{p,\lambda}\|f\|_{\M^{p,\lambda}(\R)},
\qquad 1<p<\infty,
\]
is classical; see \cite{ChiarenzaFrasca1987}.  The case $\lambda=0$ is the usual $L^p$ theorem.

\begin{lemma}\label{lem:kernel-max}
Let $k\in L^1(\R)$ satisfy
\[
|k(t)|\le C_N(1+|t|)^{-N}
\]
for some $N>2$, and put $k_h(t)=h^{-1}k(t/h)$.  Then for every locally integrable Banach-valued function $f$,
\[
\|k_h*f(t)\|_X\le C M(\|f\|_X)(t)
\]
for almost every $t$, uniformly in $h>0$.
\end{lemma}

\begin{proof}
Write $F=\|f\|_X$.  Since
\[
\|k_h*f(t)\|_X
\le \int_{\R}|k_h(s)|F(t-s)\,ds,
\]
it is enough to estimate the scalar convolution on the right.  Decompose $\R$ into
\[
A_0:=\{|s|<h\},
\qquad
A_\ell:=\{2^\ell h\le |s|<2^{\ell+1}h\},\quad \ell\ge0.
\]
On $A_0$ we have $|k_h(s)|\le C h^{-1}$, and hence
\[
\int_{A_0}|k_h(s)|F(t-s)\,ds
\le C h^{-1}\int_{t-h}^{t+h}F(\tau)\,d\tau
\le C MF(t).
\]
On $A_\ell$, the decay of $k$ gives
\[
|k_h(s)|
=h^{-1}|k(s/h)|
\le C h^{-1}2^{-N\ell}.
\]
Moreover $t-A_\ell\subset(t-2^{\ell+1}h,t+2^{\ell+1}h)$, so
\[
\int_{A_\ell}F(t-s)\,ds
\le C2^\ell h\,MF(t).
\]
Therefore
\[
\int_{A_\ell}|k_h(s)|F(t-s)\,ds
\le C2^{-(N-1)\ell}MF(t).
\]
The geometric series is summable because $N>2$, and the resulting constant is independent of $h$.  This proves the claimed pointwise bound.
\end{proof}

\section{The geometry-free derivative theorem on the line}

\begin{theorem}\label{thm:whole-line}
Let $E_0\hookrightarrow E$ continuously, $m\ge2$, $1<p<\infty$, $0\le\lambda<1$, and $0<j<m$.  Then
\[
D_t^j:\DM^{m,p,\lambda}(\R;E_0,E)
\longrightarrow
\M^{p,\lambda}(\R;X_{j/m})
\]
is bounded.  More precisely,
\[
\|D_t^j u\|_{\M^{p,\lambda}(\R;X_{j/m})}
\le C
\bigl(
\|u\|_{\M^{p,\lambda}(\R;E_0)}
+\|D_t^m u\|_{\M^{p,\lambda}(\R;E)}
\bigr).
\]
No geometric assumption on the Banach spaces is required.
\end{theorem}

\begin{proof}
Set $\theta=j/m$.  The argument is a scale-by-scale decomposition of $D_t^j u$ into an $E_0$-part and an $E$-part, chosen so that the Peetre parameter exactly matches the time-frequency scale.

For $\eta\in\R$ define
\[
a(\eta):=\frac{(i\eta)^j}{1+\eta^{2m}},
\qquad
c(\eta):=i^{j-m}\frac{\eta^{j+m}}{1+\eta^{2m}}.
\]
Because $0<j<m$, both symbols are proper rational functions.  Their poles are the nonreal roots of $-1$, so the inverse Fourier transforms $k_a$ and $k_c$ are finite linear combinations of exponentially decaying functions on the two half-lines.  In particular, for every $N$,
\[
|k_a(t)|+|k_c(t)|\le C_N(1+|t|)^{-N}.
\]

Fix $h>0$.  Let $A_h$ and $C_h$ be the convolution operators whose multipliers are
\[
h^{-j}a(h\xi)
=\frac{(i\xi)^j}{1+(h\xi)^{2m}},
\]
and
\[
h^{m-j}c(h\xi),
\]
respectively.  If the second multiplier is applied to $D_t^m u$, then
\[
h^{m-j}c(h\xi)(i\xi)^m
=\frac{(i\xi)^j(h\xi)^{2m}}{1+(h\xi)^{2m}}.
\]
Consequently the elementary identity
\[
1=\frac1{1+(h\xi)^{2m}}+
\frac{(h\xi)^{2m}}{1+(h\xi)^{2m}}
\]
gives, first for smooth compactly supported functions and then by distributional continuity,
\begin{equation}\label{eq:rational-splitting}
D_t^j u=A_hu+C_hD_t^m u.
\end{equation}
The first term is $E_0$-valued, because $A_h$ is a scalar convolution operator acting on $u$; the second is $E$-valued.

The kernels of $A_h$ and $C_h$ are, respectively,
\[
h^{-j-1}k_a(t/h),
\qquad
h^{m-j-1}k_c(t/h).
\]
Apply Lemma~\ref{lem:kernel-max} to
\[
f_0(t):=\|u(t)\|_{E_0},
\qquad
f_m(t):=\|D_t^m u(t)\|_E.
\]
For almost every $t$,
\begin{equation}\label{eq:Ah-Ch-pointwise}
\|A_hu(t)\|_{E_0}
\le Ch^{-j}Mf_0(t),
\qquad
\|C_hD_t^m u(t)\|_E
\le Ch^{m-j}Mf_m(t).
\end{equation}

Now use \eqref{eq:rational-splitting} as an admissible decomposition in the $K$-functional.  At the specific Peetre parameter
\[
r=h^{-m},
\]
we obtain
\[
\begin{aligned}
K(r,D_t^j u(t);E_0,E)
&\le \|A_hu(t)\|_{E_0}
+r\|C_hD_t^m u(t)\|_E\\
&\le Ch^{-j}Mf_0(t)
+h^{-m}Ch^{m-j}Mf_m(t)\\
&\le Ch^{-j}\bigl(Mf_0(t)+Mf_m(t)\bigr).
\end{aligned}
\]
Since $r=h^{-m}$ runs through all positive numbers as $h$ runs through $(0,\infty)$ and
\[
r^{-\theta}=h^{m\theta}=h^j,
\]
we may multiply the last inequality by $r^{-\theta}$ and take the supremum over $r>0$.  This yields the pointwise interpolation estimate
\begin{equation}\label{eq:pointwise-weak-interpolation}
\|D_t^j u(t)\|_{X_{j/m}}
\le C\bigl(Mf_0(t)+Mf_m(t)\bigr).
\end{equation}
Finally, the scalar Hardy--Littlewood maximal operator is bounded on $\M^{p,\lambda}(\R)$ for $1<p<\infty$ and $0\le\lambda<1$.  Taking the Morrey norm in \eqref{eq:pointwise-weak-interpolation} gives
\[
\|D_t^j u\|_{\M^{p,\lambda}(X_{j/m})}
\le C\bigl(
\|u\|_{\M^{p,\lambda}(E_0)}
+\|D_t^m u\|_{\M^{p,\lambda}(E)}
\bigr),
\]
which proves the theorem.
\end{proof}

\begin{corollary}\label{cor:multiplicative}
Under the assumptions of Theorem~\ref{thm:whole-line},
\[
\boxed{
\|D_t^j u\|_{\M^{p,\lambda}(X_{j/m})}
\le C
\|u\|_{\M^{p,\lambda}(E_0)}^{1-j/m}
\|D_t^m u\|_{\M^{p,\lambda}(E)}^{j/m}.
}
\]
If one of the two factors on the right vanishes, the assertion is understood by continuity.
\end{corollary}

\begin{proof}
Put
\[
a=\frac{1-\lambda}{p},
\qquad
A:=\|u\|_{\M^{p,\lambda}(E_0)},
\qquad
B:=\|D_t^m u\|_{\M^{p,\lambda}(E)}.
\]
For $s>0$ define $u_s(t)=u(st)$.  A change of variables on each interval in the Morrey supremum gives
\[
\|u_s\|_{\M^{p,\lambda}(E_0)}=s^{-a}A.
\]
Since $D_t^m u_s(t)=s^m(D_t^m u)(st)$ and $D_t^j u_s(t)=s^j(D_t^j u)(st)$,
\[
\|D_t^m u_s\|_{\M^{p,\lambda}(E)}=s^{m-a}B,
\qquad
\|D_t^j u_s\|_{\M^{p,\lambda}(X_{j/m})}
=s^{j-a}\|D_t^j u\|_{\M^{p,\lambda}(X_{j/m})}.
\]
Apply Theorem~\ref{thm:whole-line} to $u_s$ and divide by $s^{j-a}$:
\[
\|D_t^j u\|_{\M^{p,\lambda}(X_{j/m})}
\le C\bigl(s^{-j}A+s^{m-j}B\bigr).
\]
If $A,B>0$, choose $s^m=A/B$.  The two terms then have the same homogeneity and give
\[
s^{-j}A=A^{1-j/m}B^{j/m},
\qquad
s^{m-j}B=A^{1-j/m}B^{j/m}.
\]
This proves the stated estimate.  If $A=0$ or $B=0$, approximate by replacing the vanishing quantity by a positive number and let it tend to zero; equivalently, use the preceding inequality and optimize in the limiting direction $s\to0$ or $s\to\infty$.
\end{proof}

\begin{corollary}\label{cor:weaker-morrey}
Let the assumptions of Theorem~\ref{thm:whole-line} hold, and let $J\subset\mathbb R$ be a bounded interval.  If $1\le q\le p$, $0\le\mu<1$, and
\[
\delta:=\frac{1-\mu}{q}-\frac{1-\lambda}{p}\ge0,
\]
then
\[
\|D_t^j u\|_{\mathcal M^{q,\mu}(J;X_{j/m})}
\le
C|J|^\delta
\bigl(
\|u\|_{\M^{p,\lambda}(\R;E_0)}
+\|D_t^m u\|_{\M^{p,\lambda}(\R;E)}
\bigr).
\]
\end{corollary}

\begin{proof}
For every interval $K\subset J$, H\"older's inequality gives
\[
|K|^{-\mu/q}\|f\|_{L^q(K)}
\le
|K|^\delta
\|f\|_{\M^{p,\lambda}(\R)}
\le |J|^\delta\|f\|_{\M^{p,\lambda}(\R)}.
\]
Apply this with $f(t)=\|D_t^j u(t)\|_{X_{j/m}}$ and use Theorem~\ref{thm:whole-line}.
\end{proof}

\begin{remark}
The weak real-interpolation index $\infty$ is part of the geometry-free conclusion.  The next proposition shows that it is forced in the class of arbitrary Banach couples, rather than being an artifact of the proof.  Section~\ref{sec:hilbert} gives a Hilbert-scale sharpening when additional geometry is available.
\end{remark}

\begin{proposition}[Sharpness of the weak fine index]\label{prop:weak-index-sharp}
Let $m\ge2$ and $0<j<m$.  For every $1<p<\infty$, $0\le\lambda<1$, and every finite $1\le q<\infty$, there exist Banach spaces $E_0\hookrightarrow E$ and a function
\[
u\in\DM^{m,p,\lambda}(\mathbb R;E_0,E)
\]
such that, on a nonempty time interval,
\[
D_t^j u(t)\in(E_0,E)_{j/m,\infty}
\qquad\text{but}\qquad
D_t^j u(t)\notin(E_0,E)_{j/m,q}.
\]
Consequently the target $(E_0,E)_{j/m,\infty}$ in Theorem~\ref{thm:whole-line} cannot, in general, be replaced by any finite fine-index space at the same interpolation level.
\end{proposition}

\begin{proof}
It is enough to work over the complex field.  Let
\[
E=c_0,
\qquad
E_0=\left\{x=(x_n)\in c_0:\ (2^{mn}x_n)_{n\ge1}\in c_0\right\},
\]
with the natural norms.  Then $E_0\hookrightarrow E$.  For this diagonal couple one has
\begin{equation}\label{eq:diag-K}
K(r,x;E_0,E)\asymp
\sup_{n\ge1}\min\{2^{mn},r\}|x_n|.
\end{equation}
To see this, write $x=y+z$ with $y\in E_0$ and $z\in E$.  For every coordinate $n$,
\[
\min\{2^{mn},r\}|x_n|
\le 2^{mn}|y_n|+r|z_n|
\le \|y\|_{E_0}+r\|z\|_E.
\]
Taking the supremum in $n$ and then the infimum over all decompositions gives the lower estimate in \eqref{eq:diag-K}.  Conversely, set
\[
y_n=x_n\mathbf1_{\{2^{mn}\le r\}},
\qquad
z_n=x_n\mathbf1_{\{2^{mn}>r\}}.
\]
Then $y\in E_0$, $z\in E$, and
\[
\|y\|_{E_0}+r\|z\|_E
\le 2\sup_n\min\{2^{mn},r\}|x_n|.
\]
This proves \eqref{eq:diag-K} with constants independent of $r$ and $x$.

Choose a decreasing sequence $\varepsilon_n\downarrow0$ such that
\[
\sum_{n\ge2}\varepsilon_n^q=\infty
\qquad\text{for every finite }q;
\]
for instance $\varepsilon_n=1/\log(n+1)$.  Let $\chi\in C_c^\infty(\mathbb R)$ satisfy $\chi=1$ on $(0,1)$ and set
\[
u(t)=\bigl(\varepsilon_n2^{-mn}\chi(t)e^{i2^nt}\bigr)_{n\ge1}.
\]
For each $t$, the weighted coordinates defining the $E_0$ norm tend to zero, and the $E_0$-norm of $u(t)$ is uniformly bounded.  Moreover the map is strongly measurable (indeed continuous) because the $c_0$ tails are uniformly small.  Leibniz' rule gives
\[
D_t^m\bigl(\varepsilon_n2^{-mn}\chi(t)e^{i2^nt}\bigr)
=
\varepsilon_n\sum_{k=0}^m c_{m,k}2^{-kn}\chi^{(k)}(t)e^{i2^nt},
\]
so $D_t^m u$ is a bounded, compactly supported $c_0$-valued function.  Hence $u\in\DM^{m,p,\lambda}(\mathbb R;E_0,E)$ for every stated $p,\lambda$.

On $(0,1)$,
\[
D_t^j u(t)=i^j\bigl(\varepsilon_n2^{-n(m-j)}e^{i2^nt}\bigr)_{n\ge1}=:y(t).
\]
Put $\theta=j/m$.  Formula~\eqref{eq:diag-K} implies
\[
K(r,y(t);E_0,E)\asymp
\sup_{n\ge1}\varepsilon_n\min\{2^{mn},r\}2^{-n(m-j)}.
\]
The elementary inequality
\[
\min\{A,r\}A^{-(1-\theta)}\le r^\theta
\]
shows that $\sup_{r>0}r^{-\theta}K(r,y(t))<\infty$.  On the other hand, if
\[
2^{mn}\le r\le 2^{mn+m/2},
\]
then the $n$th coordinate gives
\[
r^{-\theta}K(r,y(t))\ge c_{m,j}\varepsilon_n.
\]
These logarithmic intervals are disjoint and have fixed positive $dr/r$ length.  Consequently, for every finite $q$,
\[
\int_1^\infty\bigl[r^{-\theta}K(r,y(t))\bigr]^q\frac{dr}{r}
\ge c\sum_n\varepsilon_n^q=\infty.
\]
Thus $y(t)\in(E_0,E)_{\theta,\infty}$ and $y(t)\notin(E_0,E)_{\theta,q}$ for every finite $q$ and every $t\in(0,1)$.
\end{proof}

\section{Finite intervals}

The whole-line proof localizes through a graph-Morrey extension.  We record the extension because uniformity under shrinking subintervals is exactly where naive localization arguments can fail.

\begin{lemma}\label{lem:extension}
Let $E_0\hookrightarrow E$, $m\ge1$, $1<p<\infty$, $0\le\lambda<1$, and $T>0$.  There is a bounded linear extension
\[
\mathscr E_T:\DM^{m,p,\lambda}(0,T;E_0,E)
\longrightarrow
\DM^{m,p,\lambda}(\R;E_0,E)
\]
such that $(\mathscr E_Tu)|_{(0,T)}=u$.  Its norm depends only on $m,p,\lambda,T$ and the embedding $E_0\hookrightarrow E$.
\end{lemma}

\begin{proof}
We separate the construction into reflection and cutoff steps.

First, because $(0,T)$ itself is one of the intervals allowed in the Morrey norm,
\[
\|u\|_{L^p(0,T;E_0)}
\le T^{\lambda/p}\|u\|_{\M^{p,\lambda}(0,T;E_0)},
\]
and similarly
\[
\|D_t^m u\|_{L^p(0,T;E)}
\le T^{\lambda/p}\|D_t^m u\|_{\M^{p,\lambda}(0,T;E)}.
\]
Since $E_0\hookrightarrow E$, it follows that $u,D_t^m u\in L^p(0,T;E)$.  The one-dimensional distributional characterization of Sobolev spaces therefore yields
\[
u\in W^{m,p}(0,T;E).
\]
In particular, $D_t^k u$ has the standard endpoint traces in $E$ for $0\le k<m$.

Choose coefficients $c_1,\ldots,c_{m+1}$ solving the finite Vandermonde system
\[
\sum_{\ell=1}^{m+1}c_\ell(-\ell)^k=1,
\qquad k=0,\ldots,m-1.
\]
For $-T/(m+1)<t<0$ define
\[
(R_0u)(t)=\sum_{\ell=1}^{m+1}c_\ell u(-\ell t).
\]
Differentiating $k<m$ times and taking $t\uparrow0$ gives
\[
D_t^k(R_0u)(0^-)
=\sum_{\ell=1}^{m+1}c_\ell(-\ell)^kD_t^k u(0^+)
=D_t^k u(0^+),
\]
so no distributional boundary masses are created when the pieces are glued.  Reflect in the same way across $T$ to obtain a function $\widetilde u$ on a fixed larger interval $I_T\supset[0,T]$.

We record the Morrey scaling used for the reflected pieces.  If $g_\ell(t)=g(\ell t)$, then for every interval $J$,
\[
|J|^{-\lambda/p}\|g_\ell\|_{L^p(J)}
=\ell^{(\lambda-1)/p}
|\ell J|^{-\lambda/p}\|g\|_{L^p(\ell J)}.
\]
Thus dilation by a fixed integer $\ell$ changes the Morrey norm only by the fixed factor $\ell^{(\lambda-1)/p}$.  For the highest derivative there is in addition the factor $\ell^m$.  Since only the finitely many integers $1,\ldots,m+1$ occur, the reflection operator satisfies
\[
\|\widetilde u\|_{\M^{p,\lambda}(I_T;E_0)}
+\|D_t^m\widetilde u\|_{\M^{p,\lambda}(I_T;E)}
\le C\|u\|_{\DM^{m,p,\lambda}(0,T;E_0,E)}.
\]

It remains to turn the finite-interval extension into a whole-line extension.  Let $\eta\in C_c^\infty(I_T)$ be equal to one on $[0,T]$ and put
\[
\mathscr E_Tu:=\eta\widetilde u,
\]
extended by zero outside $I_T$.  The $E_0$-valued Morrey norm of $\eta\widetilde u$ is immediately controlled by that of $\widetilde u$.  For the highest derivative, Leibniz' formula gives
\[
D_t^m(\eta\widetilde u)
=\eta D_t^m\widetilde u
+\sum_{k=0}^{m-1}\binom{m}{k}
D_t^{m-k}\eta\,D_t^k\widetilde u.
\]
The first term is already controlled in $\M^{p,\lambda}(E)$.  For the remaining terms we use the fixed-interval Sobolev estimate
\[
\sum_{k=0}^{m-1}
\|D_t^k\widetilde u\|_{L^\infty(I_T;E)}
\le C_{m,p,T}
\bigl(
\|\widetilde u\|_{L^p(I_T;E)}
+\|D_t^m\widetilde u\|_{L^p(I_T;E)}
\bigr).
\]
The right-hand side is bounded by the graph-Morrey norm established above.  Each cutoff term is therefore a bounded $E$-valued function supported in a fixed compact set.  Such a function belongs to $\M^{p,\lambda}(\R;E)$, with norm bounded by a constant depending only on the size of that fixed support.  Hence
\[
\|\mathscr E_Tu\|_{\DM^{m,p,\lambda}(\R;E_0,E)}
\le C\|u\|_{\DM^{m,p,\lambda}(0,T;E_0,E)}.
\]
Linearity follows from the construction, and $\mathscr E_Tu=u$ on $(0,T)$ because $\eta\equiv1$ there.
\end{proof}

\begin{theorem}\label{thm:finite-interval}
Let $E_0\hookrightarrow E$, $m\ge2$, $1<p<\infty$, $0\le\lambda<1$, and $0<j<m$.  Then
\[
D_t^j:\DM^{m,p,\lambda}(0,T;E_0,E)
\longrightarrow
\M^{p,\lambda}\bigl(0,T;X_{j/m}\bigr)
\]
is bounded.
\end{theorem}

\begin{proof}
Let $U=\mathscr E_Tu$.  By Lemma~\ref{lem:extension},
\[
\|U\|_{\DM^{m,p,\lambda}(\R;E_0,E)}
\le C_T\|u\|_{\DM^{m,p,\lambda}(0,T;E_0,E)}.
\]
Theorem~\ref{thm:whole-line} applied to $U$ gives
\[
\|D_t^jU\|_{\M^{p,\lambda}(\R;X_{j/m})}
\le C\|U\|_{\DM^{m,p,\lambda}(\R;E_0,E)}.
\]
Since $U=u$ on $(0,T)$, the distributional derivatives agree there.  Restriction can only decrease the Morrey norm, so
\[
\|D_t^j u\|_{\M^{p,\lambda}(0,T;X_{j/m})}
\le CC_T\|u\|_{\DM^{m,p,\lambda}(0,T;E_0,E)}.
\]
\end{proof}

\section{Derivative traces and mixed time regularity}\label{sec:traces}

We now assume $0<\lambda<1$ and use the exact first-order Morrey trace theorem from \cite{ShahmurovTrace2026}.  Set
\[
a:=\frac{1-\lambda}{p},
\qquad
\vartheta:=1-a.
\]
For $0\le j<m$ define
\[
\tau_j:=\frac{j+a}{m}.
\]

\begin{theorem}\label{thm:derivative-traces}
Let the assumptions of Theorem~\ref{thm:finite-interval} hold with $0<\lambda<1$.  For each $0\le j<m$,
\[
D_t^j u\in
L^\infty\bigl(0,T;X_{\tau_j}\bigr)
\cap
C^{0,\vartheta}\bigl([0,T];X_{(j+1)/m}\bigr),
\]
and
\[
D_t^j u(0),\ D_t^j u(T)\in X_{\tau_j}.
\]
Moreover,
\[
\|D_t^j u\|_{L^\infty(X_{\tau_j})}
+[D_t^j u]_{C^{0,\vartheta}(X_{(j+1)/m})}
\le C\|u\|_{\DM^{m,p,\lambda}}.
\]
\end{theorem}

\begin{proof}
Set $v=D_t^j u$.  We first verify that $v$ lies in the first-order Morrey evolution class to which the exact trace theorem applies.

If $j=0$, then
\[
v=u\in\M^{p,\lambda}(0,T;X_0)
\]
by definition.  If $j>0$, Theorem~\ref{thm:finite-interval} gives
\[
v\in\M^{p,\lambda}(0,T;X_{j/m}).
\]
For the derivative of $v$, if $j+1<m$ we apply Theorem~\ref{thm:finite-interval} with $j+1$ and obtain
\[
v'=D_t^{j+1}u
\in\M^{p,\lambda}(0,T;X_{(j+1)/m}).
\]
When $j=m-1$, this last statement is simply the defining assumption
\[
v'=D_t^m u\in\M^{p,\lambda}(0,T;E)
=\M^{p,\lambda}(0,T;X_1).
\]
Thus in all cases
\[
v\in\M^{p,\lambda}(0,T;Y_1),
\qquad
v'\in\M^{p,\lambda}(0,T;Y_0),
\]
with
\[
Y_1:=X_{j/m},
\qquad
Y_0:=X_{(j+1)/m},
\qquad
Y_1\hookrightarrow Y_0.
\]

The exact first-order Morrey trace theorem from \cite{ShahmurovTrace2026} states that for
\[
a=\frac{1-\lambda}{p},
\qquad
\vartheta=1-a,
\]
such a function satisfies
\[
v\in L^\infty\bigl(0,T;(Y_0,Y_1)_{\vartheta,\infty}\bigr)
\cap C^{0,\vartheta}([0,T];Y_0),
\]
with the corresponding norm bounded by
\[
C\bigl(
\|v\|_{\M^{p,\lambda}(Y_1)}
+\|v'\|_{\M^{p,\lambda}(Y_0)}
\bigr).
\]
The preceding derivative estimates bound this quantity by $C\|u\|_{\DM^{m,p,\lambda}}$.

It remains only to identify the trace space.  By real reiteration,
\[
(Y_0,Y_1)_{\vartheta,\infty}
=(X_{(j+1)/m},X_{j/m})_{\vartheta,\infty}
=(E_0,E)_{\tau_j,\infty},
\]
where the interpolation coordinate is
\[
(1-\vartheta)\frac{j+1}{m}
+\vartheta\frac jm
=a\frac{j+1}{m}+(1-a)\frac jm
=\frac{j+a}{m}=\tau_j.
\]
Hence the $L^\infty$ and endpoint-trace conclusions take values in $X_{\tau_j}$, while the H\"older estimate takes values in $Y_0=X_{(j+1)/m}$.  This is exactly the asserted statement.
\end{proof}

\begin{theorem}[Exact derivative trace range]\label{thm:exact-derivative-trace}
Let $0<\lambda<1$ and $0\le j<m$.  Then
\[
\operatorname{Tr}_0 D_t^j\DM^{m,p,\lambda}(0,T;E_0,E)
=(E_0,E)_{\tau_j,\infty}=X_{\tau_j}.
\]
More precisely, the trace map $u\mapsto D_t^j u(0)$ is bounded into $X_{\tau_j}$, and for every $x\in X_{\tau_j}$ there exists $u\in\DM^{m,p,\lambda}(0,T;E_0,E)$ such that
\[
D_t^j u(0)=x,
\qquad
\|u\|_{\DM^{m,p,\lambda}}
\le C\|x\|_{X_{\tau_j}}.
\]
Thus the extension may be chosen as a bounded assignment; linearity is not required.
\end{theorem}

\begin{proof}
The boundedness of the trace map is already contained in Theorem~\ref{thm:derivative-traces}; only surjectivity requires construction.

Fix $x\in X_{\tau_j}$ and put
\[
a=\frac{1-\lambda}{p},
\qquad
\tau=\tau_j=\frac{j+a}{m},
\qquad
M:=\|x\|_{(E_0,E)_{\tau,\infty}}.
\]
Because $0<a<1$ and $0\le j<m$, we have $0<\tau<1$.  Choose $T_0<T/2$ and dyadic scales
\[
h_n=T_0 2^{-n},\qquad n\ge0.
\]
For the Peetre parameter $r_n=h_n^{-m}$, the definition of the $K$-functional gives a decomposition
\[
x=x_n+e_n,
\qquad
x_n\in E_0,
\quad e_n\in E,
\]
which may be chosen so that
\[
\|x_n\|_{E_0}+h_n^{-m}\|e_n\|_E
\le 2K(h_n^{-m},x).
\]
Since $x\in(E_0,E)_{\tau,\infty}$,
\[
K(h_n^{-m},x)
\le M(h_n^{-m})^\tau
=Mh_n^{-m\tau}.
\]
Using $m\tau=j+a$ we therefore obtain
\begin{equation}\label{eq:trace-approximants}
\|x_n\|_{E_0}\le 2M h_n^{-j-a},
\qquad
\|e_n\|_E\le 2M h_n^{m-j-a}.
\end{equation}
The second exponent is positive, because $m-j-a\ge1-a>0$.

We now turn these discrete approximants into a smooth function of time.  Choose a nonnegative $\psi\in C_c^\infty((1/4,4))$, positive on $[1/2,2]$, such that the dyadic family
\[
\psi_n(t):=
\frac{\psi(t/h_n)}{\sum_{k\ge0}\psi(t/h_k)}
\]
forms a partition of unity for $0<t<T_1$ with some $T_1<T_0$.  After a harmless smooth completion on $[T_1,T_0]$, we may assume that
\[
\sum_n\psi_n(t)=1\quad(0<t<T_1),
\]
that at each $t$ only a fixed number of indices occur, and that
\[
\operatorname{supp}\psi_n\subset[c h_n,C h_n],
\qquad
|D_t^k\psi_n(t)|\le C_k h_n^{-k}.
\]
Finally multiply the completed family by a cutoff supported in $(0,T_0)$ and equal to one on $(0,T_1)$.  Define
\[
v(t):=\sum_n\psi_n(t)x_n.
\]
The sum is locally finite for $t>0$, hence $v$ is smooth there with values in $E_0$.

If $0<t<T_1$, every active index satisfies $h_n\simeq t$.  The finite overlap and \eqref{eq:trace-approximants} yield
\begin{equation}\label{eq:v-strong-bound}
\|v(t)\|_{E_0}
\le CM t^{-j-a}.
\end{equation}
Moreover, because $\sum_n\psi_n(t)=1$,
\[
x-v(t)=\sum_n\psi_n(t)e_n,
\]
so
\begin{equation}\label{eq:v-approaches-x}
\|x-v(t)\|_E
\le CM t^{m-j-a}.
\end{equation}
Since $m-j-a>0$, this proves $v(t)\to x$ in $E$ as $t\downarrow0$.

Next let $k=m-j$.  Differentiating the identity $\sum_n\psi_n=1$ gives
\[
\sum_nD_t^k\psi_n=0
\qquad (0<t<T_1).
\]
Hence
\[
D_t^k v
=\sum_nD_t^k\psi_n x_n
=-\sum_nD_t^k\psi_n e_n.
\]
For an active index $h_n\simeq t$, and therefore
\[
\|D_t^k\psi_n(t)e_n\|_E
\le C h_n^{-k}Mh_n^{m-j-a}
=CMh_n^{-a}
\le CMt^{-a}.
\]
Finite overlap gives
\begin{equation}\label{eq:v-high-derivative}
\|D_t^{m-j}v(t)\|_E\le CMt^{-a},
\qquad 0<t<T_1.
\end{equation}
On $[T_1,T]$ only finitely many smooth cutoff pieces remain, so both $v$ and all derivatives occurring above are uniformly bounded there.

We now recover a function whose $j$-th derivative is $v$.  If $j=0$, simply put $u=v$.  If $j\ge1$, define
\[
u(t)=\frac{(-1)^j}{(j-1)!}
\int_t^T(s-t)^{j-1}v(s)\,ds.
\]
Because $v$ is supported in $(0,T_0)$, this Bochner integral is well defined for every $t>0$.  Repeated differentiation gives
\[
D_t^j u=v,
\qquad
D_t^m u=D_t^{m-j}v.
\]
For the strong-space norm, \eqref{eq:v-strong-bound} implies, for $0<t<T_1$,
\[
\|u(t)\|_{E_0}
\le CM\int_t^{T_0}(s-t)^{j-1}s^{-j-a}\,ds.
\]
With $s=t(1+y)$,
\[
\int_t^{T_0}(s-t)^{j-1}s^{-j-a}\,ds
\le t^{-a}\int_0^\infty y^{j-1}(1+y)^{-j-a}\,dy
\le C_{j,a}t^{-a},
\]
where the last integral converges because $a>0$.  Thus, also in the case $j=0$ by \eqref{eq:v-strong-bound},
\[
\|u(t)\|_{E_0}\le CM(1+t^{-a}),
\qquad
\|D_t^m u(t)\|_E\le CM(1+t^{-a}).
\]

It remains to check the Morrey integrability of the borderline power.  Since
\[
ap=1-\lambda,
\]
for any interval $J\subset(0,T)$ of length $\ell$ the monotonicity of $t^{-(1-\lambda)}$ gives
\[
\int_J t^{-ap}\,dt
\le \int_0^\ell t^{-(1-\lambda)}\,dt
=\frac{\ell^\lambda}{\lambda}.
\]
Hence
\[
\|t^{-a}\|_{\M^{p,\lambda}(0,T)}\le \lambda^{-1/p}.
\]
The constant function also belongs to the same finite-interval Morrey space.  Therefore
\[
\|u\|_{\DM^{m,p,\lambda}(0,T;E_0,E)}
\le C M.
\]

Finally, $D_t^j u(t)=v(t)\to x$ in $E$ by \eqref{eq:v-approaches-x}.  On the other hand, Theorem~\ref{thm:derivative-traces} gives a continuous representative of $v$ at $t=0$ in the weaker space $X_{(j+1)/m}$ and a trace $v(0)\in X_{\tau_j}$.  Since $X_{(j+1)/m}\hookrightarrow E$, the same representative converges to $v(0)$ in $E$.  Uniqueness of limits in $E$ therefore gives $v(0)=x$.  Thus $D_t^j u(0)=x$, completing the surjectivity proof.
\end{proof}

\begin{corollary}[Sharpness of the derivative trace]\label{cor:trace-sharp}
The trace level $\tau_j$ and the weak fine index $\infty$ in Theorem~\ref{thm:exact-derivative-trace} are sharp for arbitrary Banach couples.  In general one cannot replace $X_{\tau_j}$ by $(E_0,E)_{\sigma,\infty}$ with $\sigma<\tau_j$, nor by $(E_0,E)_{\tau_j,q}$ with finite $q$.
\end{corollary}

\begin{proof}
Exactness reduces both claims to the strictness of the real interpolation scale.  For example, with the diagonal couple from Proposition~\ref{prop:weak-index-sharp}, the vector
\[
x_n=2^{-mn(1-\tau_j)}
\]
satisfies $K(r,x)\asymp r^{\tau_j}$ for all large $r$.  Hence $x\in(E_0,E)_{\tau_j,\infty}$, but $x\notin(E_0,E)_{\tau_j,q}$ for finite $q$ and $x\notin(E_0,E)_{\sigma,\infty}$ whenever $\sigma<\tau_j$.
\end{proof}

\begin{corollary}[Trace-powered Morrey--Sobolev gain]\label{cor:pq-gain}
Let $0<\lambda<1$, $0\le j<m$, $1\le q<\infty$, and assume
\[
\lambda\le\mu<1.
\]
Put
\[
a=\frac{1-\lambda}{p},
\qquad
b=\frac{1-\mu}{q}.
\]
Assume
\[
0<b<a,
\]
and set
\[
\sigma:=\frac{j+a-b}{m}.
\]
Then for every $1\le r\le\infty$,
\[
D_t^j:\DM^{m,p,\lambda}(0,T;E_0,E)
\longrightarrow
\M^{q,\mu}\bigl(0,T;(E_0,E)_{\sigma,r}\bigr)
\]
is bounded.  In particular, if $\mu=\lambda$ and $q>p$, then
\[
\sigma=
\frac{j+(1-\lambda)(1/p-1/q)}{m}.
\]
At the boundary $b=a$, the same conclusion remains valid with $r=\infty$ and $\sigma=j/m$.
\end{corollary}

\begin{proof}
Let
\[
v=D_t^j u,
\qquad
s_0=\frac jm,
\qquad
s_1=\tau_j=\frac{j+a}{m}.
\]
Theorem~\ref{thm:finite-interval} and Theorem~\ref{thm:derivative-traces} give, respectively,
\begin{equation}\label{eq:pq-two-endpoints}
A(t):=\|v(t)\|_{X_{s_0}}
\in\M^{p,\lambda}(0,T),
\qquad
B:=\|v\|_{L^\infty(0,T;X_{s_1})}<\infty,
\end{equation}
and both quantities are bounded by $C\|u\|_{\DM^{m,p,\lambda}}$.

The target level lies strictly between these two levels.  Indeed, with
\[
\gamma:=\frac{b}{a}\in(0,1),
\]
we have
\[
\sigma
=\gamma s_0+(1-\gamma)s_1.
\]
We use the following elementary strict-scale interpolation estimate: for every $x\in X_{s_0}\cap X_{s_1}$ and every $1\le r\le\infty$,
\begin{equation}\label{eq:strict-scale-pointwise}
\|x\|_{(E_0,E)_{\sigma,r}}
\le C
\|x\|_{X_{s_0}}^\gamma
\|x\|_{X_{s_1}}^{1-\gamma}.
\end{equation}
For completeness, this follows directly from the $K$-functional.  If
\[
A_0=\|x\|_{X_{s_0}},
\qquad
A_1=\|x\|_{X_{s_1}},
\]
then
\[
K(\rho,x)\le
\min\{A_0\rho^{s_0},A_1\rho^{s_1}\}.
\]
Split the defining $d\rho/\rho$ integral for the $(\sigma,r)$ norm at the unique scale where the two powers are equal.  The two geometric integrals give \eqref{eq:strict-scale-pointwise}; for $r=\infty$ the same estimate follows by taking the supremum.

Apply \eqref{eq:strict-scale-pointwise} pointwise to $v(t)$ and use \eqref{eq:pq-two-endpoints}:
\[
\|v(t)\|_{(E_0,E)_{\sigma,r}}
\le C B^{1-\gamma}A(t)^\gamma.
\]
Fix an interval $J\subset(0,T)$.  Since $\mu\ge\lambda$,
\[
\gamma q
=\frac{bq}{a}
=p\frac{1-\mu}{1-\lambda}
\le p.
\]
H\"older's inequality therefore gives
\[
\begin{aligned}
\left(\int_J A(t)^{\gamma q}\,dt\right)^{1/q}
&\le
\left(\int_J A(t)^p\,dt\right)^{\gamma/p}
|J|^{1/q-\gamma/p}\\
&\le
\|A\|_{\M^{p,\lambda}}^\gamma
|J|^{\lambda\gamma/p+1/q-\gamma/p}.
\end{aligned}
\]
The exponent of $|J|$ is exactly $\mu/q$, because
\[
\frac1q-\gamma\frac{1-\lambda}{p}
=\frac1q-b
=\frac\mu q.
\]
Consequently
\[
|J|^{-\mu/q}
\|v\|_{L^q(J;(E_0,E)_{\sigma,r})}
\le C
B^{1-\gamma}
\|A\|_{\M^{p,\lambda}}^\gamma.
\]
Taking the supremum over $J$ and using \eqref{eq:pq-two-endpoints} proves the asserted estimate.

If $b=a$, then $\gamma=1$ and $\sigma=s_0$.  In that case no finite fine-index improvement at the critical level is available in general, by Proposition~\ref{prop:weak-index-sharp}; however Theorem~\ref{thm:finite-interval} and Corollary~\ref{cor:weaker-morrey} apply because $b=a$ implies $q\le p$ and the scalar Morrey exponent shift is zero.  They give the stated boundary conclusion with $r=\infty$.
\end{proof}

\begin{proposition}[Exact Morrey-parameter threshold]\label{prop:mu-threshold}
Fix $0<\lambda<1$, $0\le j<m$, $1\le q<\infty$, $0\le\mu<1$, and $1\le r\le\infty$.  Put
\[
a=\frac{1-\lambda}{p},
\qquad
b=\frac{1-\mu}{q},
\qquad
\sigma=\frac{j+a-b}{m},
\]
and assume the strict-gap condition
\[
0<b<a.
\]
Then the following are equivalent:
\begin{enumerate}[label=\textup{(\roman*)}]
\item for every continuously embedded Banach couple $E_0\hookrightarrow E$, the map
\[
D_t^j:\DM^{m,p,\lambda}(0,T;E_0,E)
\longrightarrow
\M^{q,\mu}\bigl(0,T;(E_0,E)_{\sigma,r}\bigr)
\]
is bounded;
\item $\mu\ge\lambda$.
\end{enumerate}
Moreover, if $\mu<\lambda$, failure already occurs for one fixed densely and compactly embedded Hilbert couple.
\end{proposition}

\begin{proof}
The implication \textup{(ii)}$\Rightarrow$\textup{(i)} is Corollary~\ref{cor:pq-gain}.  We prove necessity.

An affine change of the time variable only changes the constants in the norms, so it is enough to work on $(0,1)$.  Let
\[
E=\ell^2,
\qquad
E_0=\left\{x=(x_k)_{k\ge1}:\sum_{k\ge1}2^{2mk}|x_k|^2<\infty\right\},
\]
with
\[
\|x\|_{E_0}^2=\sum_{k\ge1}2^{2mk}|x_k|^2.
\]
Then $E_0\hookc E$ densely and compactly.  Denote by $e_n$ the $n$th coordinate vector and set
\[
h_n=2^{-n}.
\]
Thus
\[
\|e_n\|_E=1,
\qquad
\|e_n\|_{E_0}=h_n^{-m}.
\]

Choose a nonzero $\phi\in C_c^\infty((-1/10,1/10))$.  Let
\[
N_n=\lfloor c_0h_n^{-\lambda}\rfloor
\]
with $c_0>0$ small and choose points $t_{1,n},\dots,t_{N_n,n}$ in a fixed interval $J_*\Subset(0,1)$, equally spaced with separation
\[
d_n\asymp N_n^{-1}\asymp h_n^\lambda.
\]
Since $0<\lambda<1$, one has $h_n/d_n\to0$.  Hence, after discarding finitely many $n$, the supports of
\[
\phi\left(\frac{t-t_{k,n}}{h_n}\right),
\qquad 1\le k\le N_n,
\]
are pairwise disjoint and all lie in $J_*$.  Define
\begin{equation}\label{eq:threshold-pulse-def}
u_n(t)=h_n^{m-a}
\sum_{k=1}^{N_n}
\phi\left(\frac{t-t_{k,n}}{h_n}\right)e_n.
\end{equation}

We first show that the graph norms of $u_n$ remain bounded.  Since the pulses are disjoint,
\[
\|u_n(t)\|_{E_0}
=h_n^{-a}
\left|\sum_{k=1}^{N_n}
\phi\left(\frac{t-t_{k,n}}{h_n}\right)\right|,
\]
and
\[
\|D_t^m u_n(t)\|_E
=h_n^{-a}
\left|\sum_{k=1}^{N_n}
\phi^{(m)}\left(\frac{t-t_{k,n}}{h_n}\right)\right|.
\]
Thus it suffices to estimate a scalar pulse train
\[
G_n(t)=h_n^{-a}\sum_{k=1}^{N_n}
\left|\psi\left(\frac{t-t_{k,n}}{h_n}\right)\right|,
\]
where $\psi$ is either $\phi$ or $\phi^{(m)}$.

Let $J\subset(0,1)$ be an interval and write $\ell=|J|$.  Recall that
\[
ap=1-\lambda.
\]
If $\ell\le h_n$, then at most one pulse meets $J$ and
\[
\ell^{-\lambda}\int_JG_n(t)^p\,dt
\le C\ell^{-\lambda}h_n^{-(1-\lambda)}\ell
=C\left(\frac{\ell}{h_n}\right)^{1-\lambda}
\le C.
\]
If $h_n<\ell<d_n$, only a bounded number of pulses meet $J$, and therefore
\[
\ell^{-\lambda}\int_JG_n(t)^p\,dt
\le C\ell^{-\lambda}h_n^{-(1-\lambda)}h_n
=C\left(\frac{h_n}{\ell}\right)^\lambda
\le C.
\]
Finally, if $\ell\ge d_n$, the separation of the centers gives
\[
\#\{k:\operatorname{supp}\phi((\cdot-t_{k,n})/h_n)\cap J\ne\varnothing\}
\le C\frac{\ell}{d_n}.
\]
Consequently,
\[
\ell^{-\lambda}\int_JG_n(t)^p\,dt
\le
C\ell^{-\lambda}h_n^{-(1-\lambda)}h_n\frac{\ell}{d_n}
\le C\ell^{1-\lambda},
\]
because $d_n\asymp h_n^\lambda$.  This is uniformly bounded for $J\subset(0,1)$.  Taking $p$th roots in the three cases yields
\begin{equation}\label{eq:threshold-source-bound}
\sup_n\|u_n\|_{\DM^{m,p,\lambda}(0,1;E_0,E)}<\infty.
\end{equation}

It remains to compute the target size.  The coordinate projection onto $\mathbb C e_n$ is contractive on both $E_0$ and $E$, so the $K$-functional of $e_n$ is exactly the one-dimensional $K$-functional:
\[
K(\rho,e_n;E_0,E)=\min\{h_n^{-m},\rho\}.
\]
Hence, for every fixed $1\le r\le\infty$,
\begin{equation}\label{eq:threshold-coordinate-interp}
\|e_n\|_{(E_0,E)_{\sigma,r}}
\asymp_{\sigma,r} h_n^{-m(1-\sigma)}.
\end{equation}
Since
\[
m(1-\sigma)=m-j-a+b,
\]
we obtain from \eqref{eq:threshold-pulse-def} and \eqref{eq:threshold-coordinate-interp}
\[
\|D_t^ju_n(t)\|_{(E_0,E)_{\sigma,r}}
\asymp
h_n^{-b}
\left|\sum_{k=1}^{N_n}
\phi^{(j)}\left(\frac{t-t_{k,n}}{h_n}\right)\right|.
\]
Using the fixed interval $J_*$ containing all pulses and their disjointness,
\[
\begin{aligned}
\|D_t^ju_n\|_{\M^{q,\mu}(0,1;(E_0,E)_{\sigma,r})}
&\ge
|J_*|^{-\mu/q}
\|D_t^ju_n\|_{L^q(J_*;(E_0,E)_{\sigma,r})}\\
&\ge c\,h_n^{-b}(N_nh_n)^{1/q}\\
&\asymp h_n^{-b+(1-\lambda)/q}
=h_n^{(\mu-\lambda)/q}.
\end{aligned}
\]
Because $\phi^{(j)}\not\equiv0$, the constant $c$ is positive.  If $\mu<\lambda$, the last quantity tends to infinity as $n\to\infty$, contradicting boundedness in \textup{(i)} in view of \eqref{eq:threshold-source-bound}.  Thus \textup{(i)} implies $\mu\ge\lambda$.
\end{proof}

\begin{remark}
The strict condition $b<a$ for finite $r$ is essential.  If one allowed $b=a$ and $r<\infty$, the conclusion would strengthen the critical target $X_{j/m}=(E_0,E)_{j/m,\infty}$ to a finite fine index, contradicting Proposition~\ref{prop:weak-index-sharp}.  Thus the $p$-to-$q$ gain comes from combining the critical Morrey derivative estimate with the strictly smoother trace level, not from an embedding of the trace space into a stronger interpolation space.
\end{remark}

\begin{corollary}\label{cor:mixed-holder}
Let $0\le j<m$, let
\[
\tau_j<\sigma<\frac{j+1}{m},
\]
and let $1\le q\le\infty$.  Then for every
\[
0\le\beta<m(\sigma-\tau_j)
\]
there is a continuous embedding
\[
\DM^{m,p,\lambda}(0,T;E_0,E)
\hookrightarrow
C^{0,\beta}\bigl([0,T];(E_0,E)_{\sigma,q}\bigr).
\]
For $q=\infty$ the endpoint exponent $\beta=m(\sigma-\tau_j)$ is allowed.
\end{corollary}

\begin{proof}
Again set $v=D_t^j u$ and
\[
Y_0=X_{(j+1)/m},
\qquad
Y_1=X_{j/m}.
\]
As shown in the proof of Theorem~\ref{thm:derivative-traces},
\[
v\in\M^{p,\lambda}(Y_1),
\qquad
v'\in\M^{p,\lambda}(Y_0).
\]
The mixed first-order trace interpolation theorem applied to this pair says that if $0\le\eta<\vartheta$, then
\[
v\in C^{0,\vartheta-\eta}
\bigl([0,T];(Y_0,Y_1)_{\eta,\infty}\bigr).
\]
Real reiteration identifies the value space as
\[
(Y_0,Y_1)_{\eta,\infty}
=(E_0,E)_{\sigma,\infty},
\qquad
\sigma=(1-\eta)\frac{j+1}{m}+\eta\frac jm
=\frac{j+1-\eta}{m}.
\]
Since $\vartheta=1-a$ and $m\tau_j=j+a$,
\[
\vartheta-\eta
=1-a-\eta
=m\sigma-j-a
=m(\sigma-\tau_j).
\]
This proves the endpoint fine-index statement
\[
v\in C^{0,m(\sigma-\tau_j)}([0,T];(E_0,E)_{\sigma,\infty}).
\]

Now let $q<\infty$ and fix $\beta<m(\sigma-\tau_j)$.  Choose $\sigma_0$ with
\[
\tau_j<\sigma_0<\sigma,
\qquad
\beta<m(\sigma_0-\tau_j).
\]
The endpoint result at $\sigma_0$ gives
\[
v\in C^{0,\beta}([0,T];(E_0,E)_{\sigma_0,\infty}).
\]
Because $\sigma_0<\sigma$, the strict real-interpolation embedding is in the correct direction,
\[
(E_0,E)_{\sigma_0,\infty}
\hookrightarrow(E_0,E)_{\sigma,q}.
\]
Composing the H\"older estimate with this continuous embedding gives the desired finite-$q$ conclusion.
\end{proof}

\section{Compact derivative embeddings}\label{sec:compact}

We now assume
\[
E_0\hookc E.
\]
Strictly separated real-interpolation levels are then compactly embedded.  We combine this with the same-Morrey compactness theorem and the exact $C_t$ compactness theorem from \cite{ShahmurovCompact2026}.

\begin{theorem}\label{thm:compact-same-morrey}
Let $0\le j<m$, $1\le q\le\infty$, and
\[
\frac jm<\sigma<\frac{j+1}{m}.
\]
Then bounded subsets of $\DM^{m,p,\lambda}(0,T;E_0,E)$ are relatively compact in
\[
\M^{p,\lambda}\bigl(0,T;(E_0,E)_{\sigma,q}\bigr)
\]
after application of $D_t^j$.
\end{theorem}

\begin{proof}
Put
\[
Y_0=X_{j/m},\qquad
Y=(E_0,E)_{\sigma,q},\qquad
Y_1=X_{(j+1)/m}.
\]
The strict interpolation gap and $E_0\hookc E$ give
\[
Y_0\hookc Y\hookrightarrow Y_1.
\]
As in the proof of Theorem~\ref{thm:derivative-traces}, the family $v=D_t^j u$ is bounded in
\[
\M^{p,\lambda}(Y_0),
\qquad
v'\in\M^{p,\lambda}(Y_1).
\]
The same-Morrey Aubin--Lions theorem of \cite{ShahmurovCompact2026} gives compactness in $\M^{p,\lambda}(Y)$.
\end{proof}

\begin{theorem}\label{thm:compact-holder}
Let $0\le j<m$, $1\le q\le\infty$, and
\[
\tau_j<\sigma<\frac{j+1}{m}.
\]
Then for every
\[
0\le\beta<m(\sigma-\tau_j)
\]
bounded subsets of $\DM^{m,p,\lambda}(0,T;E_0,E)$ are relatively compact in
\[
C^{0,\beta}\bigl([0,T];(E_0,E)_{\sigma,q}\bigr).
\]
\end{theorem}

\begin{proof}
Let $v=D_t^j u$ and set
\[
Y_1=X_{j/m},
\qquad
Y_0=X_{(j+1)/m}.
\]
The derivative estimates already proved give a uniform bound, on bounded subsets of $\DM^{m,p,\lambda}$, for
\[
\|v\|_{\M^{p,\lambda}(Y_1)}
+\|v'\|_{\M^{p,\lambda}(Y_0)}.
\]
Because $E_0\hookc E$, strict real-interpolation levels between $E_0$ and $E$ inherit compactness; in particular the first-order pair $(Y_0,Y_1)$ satisfies the compactness hypothesis in the compact trace theorem of \cite{ShahmurovCompact2026}.

For this first-order class, the exact trace coordinate is $\vartheta=1-a$.  Interpolating between $Y_0$ and $Y_1$ at a coordinate $\eta<\vartheta$ gives, by reiteration,
\[
(Y_0,Y_1)_{\eta,q}
=(E_0,E)_{\sigma,q},
\qquad
\sigma=\frac{j+1-\eta}{m}.
\]
As in Corollary~\ref{cor:mixed-holder},
\[
\vartheta-\eta=m(\sigma-\tau_j).
\]
The compact lower-order theorem therefore yields relative compactness in every
\[
C^{0,\beta}([0,T];(E_0,E)_{\sigma,q}),
\qquad
0\le\beta<m(\sigma-\tau_j),
\]
which is the assertion.
\end{proof}

\begin{remark}
The two compactness thresholds are different.  Compactness in the same Morrey norm requires only the strict loss
\[
\sigma>\frac jm,
\]
whereas compactness in a pointwise-in-time topology begins only above
\[
\sigma>\tau_j=\frac{j+(1-\lambda)/p}{m}.
\]
The difference is the derivative version of the distinction between Aubin--Lions compactness and exact trace-space compactness.
\end{remark}

\section{Hilbert-scale sharpening}\label{sec:hilbert}

Let $H$ be a Hilbert space and let $B$ be a positive selfadjoint operator with $B\ge I$.  Put
\[
H_s:=D(B^s),\qquad \|x\|_{H_s}:=\|B^s x\|_H,
\qquad 0\le s\le1.
\]
Then
\[
H_s=(D(B),H)_{1-s,2}
\]
with equivalent norms.

\begin{lemma}\label{lem:cz-lift}
Let $T$ be an $H$-valued convolution operator on $\R$ which is bounded on $L^p(\R;H)$ and whose kernel satisfies
\[
\|K(t)\|_{\mathcal L(H)}\le C|t|^{-1},\qquad t\ne0.
\]
Then
\[
T:\M^{p,\lambda}(\R;H)\longrightarrow\M^{p,\lambda}(\R;H)
\]
is bounded for $1<p<\infty$, $0\le\lambda<1$.
\end{lemma}

\begin{proof}
Let
\[
N:=\|f\|_{\M^{p,\lambda}(\R;H)}
\]
and fix a bounded interval $J$.  Write
\[
f=f_1+f_2,
\qquad
f_1=f\mathbf1_{2J},
\quad
f_2=f\mathbf1_{(2J)^c},
\]
where $2J$ is the concentric interval of twice the length.

For the local part, the assumed $L^p$ boundedness gives
\[
\|Tf_1\|_{L^p(J;H)}
\le C_p\|f\|_{L^p(2J;H)}
\le C_p(2|J|)^{\lambda/p}N.
\]
Hence
\[
|J|^{-\lambda/p}\|Tf_1\|_{L^p(J;H)}
\le C N.
\]

For the far part, if $t\in J$ and
\[
A_k:=2^{k+1}J\setminus2^kJ,
\qquad k\ge1,
\]
then $|t-s|\simeq2^k|J|$ for $s\in A_k$.  The kernel bound therefore yields
\[
\|Tf_2(t)\|_H
\le C\sum_{k\ge1}(2^k|J|)^{-1}
\int_{2^{k+1}J}\|f(s)\|_H\,ds.
\]
By H\"older's inequality and the Morrey estimate,
\[
\begin{aligned}
\int_{2^{k+1}J}\|f(s)\|_H\,ds
&\le |2^{k+1}J|^{1-1/p}
\|f\|_{L^p(2^{k+1}J;H)}\\
&\le C(2^k|J|)^{1-1/p+\lambda/p}N.
\end{aligned}
\]
Thus the $k$-th annular contribution is bounded by
\[
C2^{-k(1-\lambda)/p}|J|^{-(1-\lambda)/p}N.
\]
Since $\lambda<1$, the series converges and
\[
\|Tf_2(t)\|_H
\le C|J|^{-(1-\lambda)/p}N,
\qquad t\in J.
\]
Taking the $L^p(J)$ norm gives
\[
\|Tf_2\|_{L^p(J;H)}
\le C|J|^{\lambda/p}N.
\]
Combining the local and far estimates, multiplying by $|J|^{-\lambda/p}$, and taking the supremum over $J$ proves the lemma.
\end{proof}

\begin{theorem}\label{thm:hilbert}
Let $m\ge2$, $0<j<m$, $1<p<\infty$, and $0\le\lambda<1$.  If
\[
u\in\M^{p,\lambda}(\R;D(B)),
\qquad
D_t^m u\in\M^{p,\lambda}(\R;H),
\]
then
\[
D_t^j u\in\M^{p,\lambda}\bigl(\R;D(B^{1-j/m})\bigr)
\]
and
\[
\|B^{1-j/m}D_t^j u\|_{\M^{p,\lambda}(H)}
\le C
\bigl(
\|Bu\|_{\M^{p,\lambda}(H)}
+\|D_t^m u\|_{\M^{p,\lambda}(H)}
\bigr).
\]
The same statement holds on finite intervals.
\end{theorem}

\begin{proof}
It is enough to prove the whole-line assertion.  Once that is done, apply the extension lemma with the Banach couple $(D(B),H)$ and restrict back to $(0,T)$.

Set $\theta=j/m$.  We split the target multiplier by the positive operator $B^2+|\xi|^{2m}$.  In Fourier variables,
\[
\begin{aligned}
B^{1-\theta}(i\xi)^j\widehat u
&=B^{1-\theta}(i\xi)^j
\frac{B^2}{B^2+|\xi|^{2m}}\widehat u\\
&\quad+
B^{1-\theta}(i\xi)^j
\frac{|\xi|^{2m}}{B^2+|\xi|^{2m}}\widehat u.
\end{aligned}
\]
The first term can be written as $M_0(\xi)\widehat{Bu}$ with
\[
M_0(\xi)=B^{2-\theta}(i\xi)^j(B^2+|\xi|^{2m})^{-1},
\]
and the second as $M_1(\xi)\widehat{D_t^m u}$ with
\[
M_1(\xi)=i^{j-m}B^{1-\theta}\xi^{j+m}
(B^2+\xi^{2m})^{-1}.
\]
Thus
\begin{equation}\label{eq:hilbert-multiplier-split}
B^{1-\theta}D_t^j u
=T_0(Bu)+T_1(D_t^m u),
\end{equation}
where $T_i$ are the Fourier multipliers with symbols $M_i$.

We next verify uniform $L^p$ and kernel bounds.  Let $b\ge1$ be a spectral value of $B$ and put
\[
\rho=b^{1/m}.
\]
For the first multiplier, the scalar spectral symbol is
\[
\frac{b^{2-\theta}(i\xi)^j}{b^2+|\xi|^{2m}}
=
\frac{(i\xi/\rho)^j}{1+(\xi/\rho)^{2m}}
=a(\xi/\rho),
\]
because $b^{j/m}=b^\theta$.  Similarly,
\[
i^{j-m}\frac{b^{1-\theta}\xi^{j+m}}{b^2+\xi^{2m}}
=c(\xi/\rho).
\]
Hence the scalar kernels are precisely the dilates
\[
k_{0,b}(t)=\rho k_a(\rho t),
\qquad
k_{1,b}(t)=\rho k_c(\rho t).
\]
Since $k_a$ and $k_c$ decay exponentially,
\[
\rho|k_i(\rho t)|
=|t|^{-1}\,|\rho t|\,|k_i(\rho t)|
\le \frac{C}{|t|},
\qquad t\ne0,
\]
uniformly in $\rho>0$.  The spectral theorem then gives operator-valued kernels $K_i(t)$ satisfying
\[
\|K_i(t)\|_{\mathcal L(H)}
\le\sup_{b\ge1}|k_{i,b}(t)|
\le C|t|^{-1}.
\]

The same dilation formulas also give the Mikhlin bounds.  For $\ell=0,1$,
\[
\sup_{\rho>0}\sup_{\xi\ne0}
|\xi|^\ell
\left|\frac{d^\ell}{d\xi^\ell}a(\xi/\rho)\right|<\infty,
\]
and likewise for $c$.  By functional calculus this yields the corresponding operator-norm bounds for $M_0$ and $M_1$.  Since $H$ is a Hilbert space, the one-dimensional operator-valued Mikhlin theorem applies, and both $T_0$ and $T_1$ are bounded on $L^p(\R;H)$, uniformly in the spectral scale.

Lemma~\ref{lem:cz-lift} now upgrades each $T_i$ to a bounded operator on $\M^{p,\lambda}(\R;H)$.  Taking the Morrey norm in \eqref{eq:hilbert-multiplier-split} gives
\[
\|B^{1-\theta}D_t^j u\|_{\M^{p,\lambda}(H)}
\le C\bigl(
\|Bu\|_{\M^{p,\lambda}(H)}
+\|D_t^m u\|_{\M^{p,\lambda}(H)}
\bigr).
\]
Finally, $B\ge I$ implies
\[
\|x\|_H\le\|B^{1-\theta}x\|_H,
\]
so the displayed quantity controls the full graph norm of $D(B^{1-\theta})$.  This proves the theorem.
\end{proof}

\begin{proposition}\label{prop:sharp-hilbert}
The fractional-domain exponent $1-j/m$ in Theorem~\ref{thm:hilbert} cannot be increased uniformly.
\end{proposition}

\begin{proof}
Take $H=\ell^2$ and $Be_N=N^m e_N$.  Let $0\ne\psi\in C_c^\infty(\R)$ and set
\[
u_N(t)=\psi(Nt)e_N.
\]
If $a=(1-\lambda)/p$, time dilation gives
\[
\|Bu_N\|_{\M^{p,\lambda}(H)}
+\|D_t^m u_N\|_{\M^{p,\lambda}(H)}
\asymp N^{m-a}.
\]
For any $\varepsilon>0$,
\[
\|B^{1-j/m+\varepsilon}D_t^j u_N\|_{\M^{p,\lambda}(H)}
\asymp N^{m+m\varepsilon-a}.
\]
The ratio diverges like $N^{m\varepsilon}$.  Hence no uniform estimate with a larger fractional-domain exponent is possible.
\end{proof}

\begin{corollary}[Hilbert Morrey--Gagliardo--Nirenberg form]\label{cor:hilbert-multiplicative}
Under the assumptions of Theorem~\ref{thm:hilbert} on $\R$, with
\[
\theta=\frac jm,
\]
one has
\[
\boxed{
\|B^{1-\theta}D_t^j u\|_{\M^{p,\lambda}(\R;H)}
\le C
\|Bu\|_{\M^{p,\lambda}(\R;H)}^{1-\theta}
\|D_t^m u\|_{\M^{p,\lambda}(\R;H)}^{\theta}.
}
\]
Consequently, for every $\varepsilon>0$,
\[
\|B^{1-\theta}D_t^j u\|_{\M^{p,\lambda}(\R;H)}
\le
\varepsilon\|D_t^m u\|_{\M^{p,\lambda}(\R;H)}
+C_\varepsilon\|Bu\|_{\M^{p,\lambda}(\R;H)}.
\]
\end{corollary}

\begin{proof}
Put $a=(1-\lambda)/p$ and $u_s(t)=u(st)$.  Time dilation gives
\[
\|Bu_s\|_{\M^{p,\lambda}}=s^{-a}\|Bu\|_{\M^{p,\lambda}},
\quad
\|D_t^m u_s\|_{\M^{p,\lambda}}
=s^{m-a}\|D_t^m u\|_{\M^{p,\lambda}},
\]
and
\[
\|B^{1-\theta}D_t^j u_s\|_{\M^{p,\lambda}}
=s^{j-a}\|B^{1-\theta}D_t^j u\|_{\M^{p,\lambda}}.
\]
Apply Theorem~\ref{thm:hilbert} to $u_s$ and divide by $s^{j-a}$:
\[
\|B^{1-\theta}D_t^j u\|_{\M^{p,\lambda}}
\le C\left(
 s^{-j}\|Bu\|_{\M^{p,\lambda}}
 +s^{m-j}\|D_t^m u\|_{\M^{p,\lambda}}
\right).
\]
Optimizing at $s^m=\|Bu\|/\|D_t^m u\|$ gives the multiplicative estimate; the cases in which one endpoint norm vanishes follow by a limiting argument.  The $\varepsilon$-form is Young's inequality.
\end{proof}

\section{Consequences and applications}\label{sec:applications}

Theorems~\ref{thm:finite-interval}, \ref{thm:exact-derivative-trace}, Proposition~\ref{prop:mu-threshold}, and Theorems~\ref{thm:compact-same-morrey}, \ref{thm:compact-holder} give a complete hierarchy for the geometry-free higher-order time class:
\[
\begin{gathered}
\text{intermediate derivative}
\longrightarrow
\text{exact derivative trace}
\longrightarrow
\text{sharp Morrey-parameter threshold}\\
\longrightarrow
\text{same-Morrey compactness}
\longrightarrow
\text{time-continuous compactness above the trace threshold}.
\end{gathered}
\]
They are directly suited to higher-order evolution equations for which a priori estimates control $u$ in a strong space and $D_t^m u$ in a weak space.

\subsection{Third-order nonlinear acoustics: an automatic middle rung}

A particularly transparent application occurs for third-order-in-time nonlinear acoustics.  The time-periodic JMGT theory in \cite{KaltenbacherPeriodic2025} treats, among others, the JMGT--Westervelt equation for acoustic pressure
\begin{equation}\label{eq:jmgt-westervelt}
\tau u_{ttt}+u_{tt}-c^2\Delta u-b\Delta u_t+\eta (u^2)_{tt}+f=0
\qquad\text{in }(0,T)\times\Omega,
\end{equation}
and the JMGT--Kuznetsov equation for the acoustic velocity potential
\begin{equation}\label{eq:jmgt-kuznetsov}
\tau u_{ttt}+u_{tt}-c^2\Delta u-b\Delta u_t
+\bigl(\widetilde\eta\,u_t^2+|\nabla u|^2\bigr)_t+f=0
\qquad\text{in }(0,T)\times\Omega.
\end{equation}
Here $\tau>0$ is the thermal relaxation time, $c>0$ the sound speed, and $b>0$ an attenuation coefficient.  In the periodic setting one imposes, in particular,
\begin{equation}\label{eq:jmgt-periodic}
u(T)=u(0),\qquad u_t(T)=u_t(0),\qquad \tau u_{tt}(T)=\tau u_{tt}(0),
\end{equation}
together with the relevant acoustic boundary conditions.  The structural point for the present paper is visible directly in \eqref{eq:jmgt-westervelt}--\eqref{eq:jmgt-kuznetsov}: the equation contains the endpoint quantities $u_{ttt}$ and $\Delta u_t$, whereas the energy and nonlinear estimates naturally involve the intermediate mixed derivative $\nabla u_{tt}$.

To isolate this mechanism in the cleanest operator setting, let $\Omega\subset\R^d$ be bounded with $C^{1,1}$ boundary, let
\[
H=L^2(\Omega),
\qquad
B=I-\Delta_D,
\]
where $-\Delta_D$ denotes the Dirichlet Laplacian.  Then $B\ge I$ is positive selfadjoint,
\[
D(B)=H^2(\Omega)\cap H_0^1(\Omega),
\qquad
D(B^{1/2})=H_0^1(\Omega)
\]
with equivalent norms.  Thus $Bu_t$ represents, up to the harmless lower-order $u_t$ term, the strong damping quantity $-\Delta u_t$, while $B^{1/2}u_{tt}$ is equivalent to one spatial derivative of the acceleration.  Apply Theorem~\ref{thm:hilbert} on $(0,T)$ to
\[
v=u_t,
\qquad m=2,
\qquad j=1.
\]
If
\[
Bu_t\in L^2(0,T;H),
\qquad
u_{ttt}\in L^2(0,T;H),
\]
then
\[
\boxed{
u_{tt}\in L^2(0,T;D(B^{1/2}))}
\]
and
\begin{equation}\label{eq:acoustic-middle-rung}
\|B^{1/2}u_{tt}\|_{L^2(0,T;H)}
\le C_T\left(
\|Bu_t\|_{L^2(0,T;H)}
+\|u_{ttt}\|_{L^2(0,T;H)}
\right).
\end{equation}
Equivalently,
\begin{equation}\label{eq:acoustic-ladder-collapse}
H^3(0,T;L^2(\Omega))
\cap H^1(0,T;H^2(\Omega)\cap H_0^1(\Omega))
\hookrightarrow
H^2(0,T;H_0^1(\Omega)).
\end{equation}
Thus the middle space--time rung is not an independent regularity requirement once the two endpoints are known.

This observation matches a regularity pattern that occurs explicitly in current JMGT theory.  In the time-periodic analysis of \cite{KaltenbacherPeriodic2025}, the medium solution class contains the three levels
\[
H^3(0,T;L^2(\Omega)),
\qquad
H^2(0,T;H^1(\Omega)),
\qquad
H^1(0,T;H^2_\Delta(\Omega)\cap L^\infty(\Omega)).
\]
For a static selfadjoint elliptic realization such as the homogeneous Dirichlet problem, \eqref{eq:acoustic-ladder-collapse} shows that the derivative part of the middle level $H_t^2H_x^1$ follows from the two endpoint levels.  The additional $L^\infty$ control and the analysis of mixed or dynamic boundary conditions in \cite{KaltenbacherPeriodic2025} are separate issues and are not supplied by the present theorem.

The same mixed derivative is an active energy quantity in the fractionally damped JMGT problem studied in \cite{MelianiSaidHouari2025}, whose model is
\begin{equation}\label{eq:fjmgt}
\tau \psi_{ttt}+(1+2k\psi_t)\psi_{tt}-c^2\Delta\psi
-\tau c^2\Delta\psi_t
-\delta\,\mathrm D_t^{1-\alpha}\Delta\psi_t=0
\qquad\text{in }(0,T)\times\Omega.
\end{equation}
Their solution energy contains $\nabla \psi_{tt}$ together with $\Delta \psi_t$, and a principal a priori step is obtained by testing the third-order equation against $-\Delta \psi_{tt}$.  Estimate \eqref{eq:acoustic-middle-rung} does not replace the nonlinear energy identity, the fractional damping estimate, or the PDE-specific control needed to establish the endpoints.  It does show, however, that once the endpoint controls $B\psi_t$ and $\psi_{ttt}$ have been established, the spatially differentiated acceleration $B^{1/2}\psi_{tt}$ requires no additional interpolation argument.  In this sense the theorem separates the model-specific energy work from the universal endpoint-to-middle regularity step.

There is a further gain in the periodic $L^2$ setting, where the multiplicative structure survives without the finite-interval polynomial obstruction.

\begin{proposition}[Periodic Hilbert interpolation and relaxation scaling]\label{prop:periodic-relaxation}
Let $B\ge I$ be positive selfadjoint on $H$, and let $v$ be $T$-periodic with
\[
v\in L^2(0,T;D(B))\cap H^2_{\rm per}(0,T;H).
\]
Then
\begin{equation}\label{eq:periodic-middle}
\|B^{1/2}v_t\|_{L^2(0,T;H)}^2
\le
\|Bv\|_{L^2(0,T;H)}
\|v_{tt}\|_{L^2(0,T;H)}.
\end{equation}
In particular, if $u_t$ is periodic, then
\begin{equation}\label{eq:periodic-acoustic}
\|B^{1/2}u_{tt}\|_{L^2}^2
\le
\|Bu_t\|_{L^2}\|u_{ttt}\|_{L^2},
\end{equation}
and for every relaxation parameter $\tau>0$,
\begin{equation}\label{eq:tau-weight}
\boxed{
\sqrt{\tau}\,\|B^{1/2}u_{tt}\|_{L^2}
\le
\|Bu_t\|_{L^2}^{1/2}
\|\tau u_{ttt}\|_{L^2}^{1/2}.
}
\end{equation}
\end{proposition}

\begin{proof}
Write the time Fourier series $v(t)=\sum_{n\in\mathbb Z}e^{i\omega_n t}v_n$, where $\omega_n=2\pi n/T$.  Parseval's identity and positivity of $B$ give
\[
\begin{aligned}
\|B^{1/2}v_t\|_{L^2}^2
&=T\sum_n \omega_n^2\langle Bv_n,v_n\rangle_H\\
&\le
\left(T\sum_n\|Bv_n\|_H^2\right)^{1/2}
\left(T\sum_n\omega_n^4\|v_n\|_H^2\right)^{1/2}\\
&=\|Bv\|_{L^2}\|v_{tt}\|_{L^2}.
\end{aligned}
\]
The remaining statements follow by taking $v=u_t$ and multiplying by $\tau^{1/2}$.
\end{proof}

Estimate \eqref{eq:tau-weight} identifies the geometric middle weight dictated by a third-order relaxation term: endpoint control of $Bu_t$ and $\tau u_{ttt}$ automatically gives $\sqrt{\tau}\,B^{1/2}u_{tt}$.  This is particularly natural for JMGT models, where the coefficient of $u_{ttt}$ is the thermal relaxation time and the limit $\tau\to0^+$ is a central analytical question; see \cite{KaltenbacherPeriodic2025,MelianiSaidHouari2025}.  The same model is now used in inverse problems for nonlinear ultrasound: \cite{KaltenbacherTomography2025} studies acoustic nonlinearity parameter tomography with the JMGT equation and boundary measurements.  Thus the reduction above applies to a regularity ladder arising in a forward model that is also used in a quantitative imaging problem.

\subsection{Related mixed-regularity bottlenecks in fluid--structure interaction}

The usefulness of compact mixed space--time estimates is not specific to acoustics.  In a recent free-interface compressible Navier--Stokes--wave analysis, \cite{KukavicaLiTuffaha2024} devote a separate section to space--time trace, interpolation, and hidden-regularity inequalities; their interior interpolation lemma and finite-interval $\varepsilon$-form are then used repeatedly in the nonlinear fixed-point estimates.  Their boundary-trace problem is different from the derivative ladder treated here, so Theorem~\ref{thm:hilbert} does not replace those trace estimates.  It does provide a direct operator-scale substitute for the interior intermediate-derivative step whenever the spatial scale is generated by a positive selfadjoint elliptic operator: endpoint control in $D(B)$ and $H$ immediately fills the fractional-domain rung $D(B^{1-j/m})$.

\subsection{Scope and limitations}

Proposition~\ref{prop:weak-index-sharp} shows that the weak fine index $\infty$ is optimal without geometric assumptions, while Proposition~\ref{prop:mu-threshold} shows that the restriction $\mu\ge\lambda$ in the scale-changing Morrey gain is also sharp and remains necessary even for a fixed compact Hilbert couple.  The Hilbert theorem shows, separately, that additional spectral geometry improves the critical weak interpolation target itself: in a fractional-domain scale one recovers $D(B^{1-j/m})$, with the sharp exponent.  More general finite-index conclusions in UMD or sectorial Banach scales are natural, but require genuinely additional multiplier hypotheses and are not part of the geometry-free theorem.

The compactness arguments are independent of the compactness claims in \cite{RagusaShakhmurov2018}; only the abstract strict-interpolation compactness and the first-order Morrey compactness theorem \cite{ShahmurovCompact2026} are used.  This separation is useful because bounded sets of standard Morrey spaces are not, in general, weakly compact, so compactness proofs should not be based on an unqualified reflexive-subsequence argument in the Morrey norm.

\section*{Statements and declarations}
\noindent\textbf{Funding.} The authors received no funding, grants, or other financial support for this work.\\
\noindent\textbf{Author contributions.} Both authors contributed equally to the conception of the work, the mathematical analysis, the development and verification of the proofs, and the writing of the manuscript. Both authors read and approved the final manuscript.\\
\noindent\textbf{Competing interests.} The authors declare no competing interests.\\
\noindent\textbf{Data availability.} No datasets were generated or analysed during the current study.\\
\noindent\textbf{Use of generative AI.} During preparation of the manuscript, the authors used OpenAI ChatGPT for editorial assistance, literature organization, and mathematical checking. All mathematical statements, proofs, and references were reviewed by the authors, who take full responsibility for the content.

\end{document}